\documentclass[12pt, a4paper]{article}
\usepackage[utf8]{inputenc}
\usepackage[T1]{fontenc}
\usepackage[english]{babel}
\usepackage{amsmath, amsfonts, amssymb, enumitem} 
\usepackage{amsthm}

\usepackage{authblk}

\usepackage{geometry}
\usepackage{graphicx} 
\usepackage{hyperref} 
\usepackage{titlesec}
\usepackage{enumitem}
\setlist[enumerate]{label=\alph*), ref=\alph*)}
\usepackage{cite}

\newtheorem{theorem}{Theorem}[section]
\newtheorem{lemma}[theorem]{Lemma}
\newtheorem{proposition}[theorem]{Proposition}
\newtheorem{definition}[theorem]{Definition}

\newtheoremstyle{prikladstyle}   
  {10pt}                       
  {10pt}                       
  {}                            
  {}                           
  {\bfseries}                  
  {.}                          
  { }                          
  {Příklad \thesection.\arabic{theorem}}

\theoremstyle{prikladstyle}

\newcommand{\Z}{\mathbb{Z}}  
\newcommand{\N}{\mathbb{N}}

\newcommand{\keywords}[1]{\noindent\textbf{Keywords:} #1}
\newcommand{\subjclass}[2]{\noindent\textbf{#1 Mathematics Subject Classification:} #2.}

\title{Using continued fractions with prescribed period for universal quadratic forms}

\author[1]{Veronika Menšíková}
\author[2]{Helena Muchová}
\affil[1]{Arcibiskupské gymnázium, Korunní 586/2, Prague, Czechia, verca.mensik@seznam.cz}
\affil[2]{Gymnázium Jana Keplera, Parléřova 118/2, Prague, Czechia, helcamuchova@gmail.com}

\begin{document} 

\maketitle

\begin{abstract} 
We study the congruence classes attained by positive integers $D$ with a prescribed period of the continued fraction of $\sqrt D$. As an application, we refine the available results on large ranks of universal quadratic forms over real quadratic fields by also imposing congruence conditions on their discriminants.

\medskip
	 
\medskip
\subjclass{2020}{11A55, 11E12, 11R11}

\end{abstract}

\section{Introduction}

In this paper, we consider continued fractions of numbers of the form $\sqrt{D}$ with a given period length. We focus on the possible residue classes of $D$ modulo $n$ and then connect our findings with universal quadratic forms over $\Z[\sqrt D]$.

In 1770, Lagrange proved the Four Square Theorem, which states that every positive integer can be written as a sum of four squares. The expression 
$w^2 + x^2 + y^2 + z^2$
is therefore a universal quadratic form over \(\mathbb{Z}\). Universal quadratic forms can, however, be studied more generally over totally real number
fields. 

Blomer and Kala \cite{blomerkala2015} discovered a connection between universal quadratic forms over rings of integers in number fields (such as $\mathbb{Z}[\sqrt{D}]$) and continued fractions.
A key tool they used is the result of Friesen \cite{friesen1988}, who studied the representation of $\sqrt{D}$ via continued fractions and determined the conditions under which one can prescribe the coefficients of $\sqrt{D}$
so that a corresponding $D$ exists. Halter-Koch \cite{halterkoch1991} extended Friesen's results by showing that one can also impose certain divisibility conditions on the coefficients.

Blomer and Kala used these theorems to construct infinitely many $D$ for which universal quadratic forms must have arbitrarily many variables. This line of research was further developed by several works, including Yatsyna \cite{yatsyna2019lower}, and Kala and Tinková \cite{kalatinkova2023}, who extended the theorem on large numbers of variables to certain cubic fields; Kala then generalized this to fields of arbitrary degree \cite{kala2023number}. More recently, Kala, Yatsyna, and Zmija \cite{kala2023realquadraticfieldsuniversal} significantly strengthened the initial result \cite{blomerkala2015}, proving that almost all quadratic fields require many variables for universal forms. For more information, see the survey article \cite{Kala_2023}.

In this paper, we refine the results of Friesen and Halter-Koch \cite{friesen1988, halterkoch1991} regarding congruence conditions. The main result is Theorem \ref{modulo n}, which states that for a continued fraction of $\sqrt{D}$ with an even period length, $D$ can take any value modulo an odd $n$, along with Theorems \ref{nej1} and \ref{nej2} concerning their applications to universal quadratic forms.

Specifically, in Section \ref{2}, we focus on continued fractions with short periods, for which we explicitly determine the form of $D$ and establish the congruence conditions. In Section \ref{3}, we extend this approach by using Friesen's results for general period lengths, and we prove Theorem \ref{4.5} and Theorem \ref{modulo n}, which generalize the results from the previous section. Finally, in Section \ref{4}, we connect these results with results on universal quadratic forms, leading to Theorems \ref{nej1} and \ref{nej2}.

Interesting open questions are whether
 Theorems \ref{modulo n} and \ref{nej2} can be generalized to even $n$ or to odd $k$. Further, although we have already shown for odd $k$ that $D$ cannot be divisible by any prime $p \equiv 3 \pmod{4}$, it may still be interesting to study the other residues modulo such primes. As a small simplification, in the paper we work only with the continued fraction of $\sqrt D$ and, accordingly, with the order $\Z[\sqrt D]$. Another interesting open problem is to extend our results to the case of continued fractions of $\frac{1+\sqrt D}2$ where $D\equiv 1\pmod 4$, so that one obtains results also for the maximal order $\Z\left[\frac{1+\sqrt D}2\right]$ in this case.

\section*{Acknowledgments}

This article is based on our high school research project. We greatly thank our advisor, Vítězslav Kala, for his help and support with the project.

\section{Continued fractions with short period} \label{2}

Let us recall that for $D \in \Z_{>1}$, which is not a perfect square, the continued fraction expansion
of $\sqrt{D}$ is
$$\sqrt{D}= [a_0, \overline{a_1, \ldots,  a_n, 2a_0}],$$
where the coefficients $a_i$ are positive integers and $(a_1, \ldots , a_n)$ is a palindrome.

Before discussing general periods using Friesen's Theorem, we focused on shorter period lengths. Let us start with the case where the period length is equal to 1.
\begin{proposition}
Let $D$ and $t$ be positive integers such that $\sqrt{D} = [t, \overline{2t}]$. Then $D = t^2 + 1$ and $D$ is not divisible by any prime number $p \equiv 3 \pmod{4}$. 

\begin{proof}
First, we determine the value of $D$ in terms of $t$:
\[
\sqrt{D} = t + \cfrac{1}{2t + \cfrac{1}{2t + \dots}}.
\]
Let $A = \sqrt{D} - t$. By the periodicity of the continued fraction expansion of the square root, after the last term of the period $2t$, the expression $A$ repeats. Thus
\[
A = \frac{1}{2t + A}.
\]
Rearranging gives
\[
A^2 + 2tA = 1.
\]
Substituting $A = \sqrt{D} - t$, we obtain
$D = t^2 + 1.$
If $p$ divides $D$, then
\[
t^2 \equiv -1 \pmod{p}.
\]
However, $-1$ is not a quadratic residue modulo any prime $p \equiv 3 \pmod{4}$, which is a contradiction.
\end{proof}

\end{proposition}

Similarly, by explicitly working with periods of length $2$ and $3$, one can straightforwardly show that for a period of length $2$, $D$ can take all residues modulo $p$. 
In contrast, for a period of length $3$, $D$ can attain all residues modulo $p=2$ or $p \equiv 1 \pmod4$, but it cannot be divisible by a prime $p \equiv 3 \pmod{4}$.

Let us further illustrate what is happening by working out the case of period $4$. First, we will establish a general formula for such $D$.

\begin{proposition}\label{period 4}
Let $D, t, u,$ and $v$ be positive integers such that $\sqrt{D} = [t, \overline{u, v, u, 2t}]$. Then 
$D = t^2 + \frac{2vut + 2t + v}{u(uv + 2)}, \quad u \neq 0, \quad uv + 2 \neq 0$ and
$$
D = \left( \frac{u\big(y(uv+2) - v^2\big) - v}{2} \right)^2 + y(uv+2) - v^2 - y
$$
for some $y \in \N$.
\end{proposition}

\begin{proof}
First, we determine the value of $D$ in terms of $t, u, v$. Let $A = \sqrt{D} - t$. By periodicity, we have
$$
A = \frac{1}{u + \frac{1}{v + \frac{1}{u + \frac{1}{2t + A}}}}.
$$
This expression can be further simplified to:
$$
A = \frac{2tuv + Auv + v + 2t + A}{2tu^2v + Au^2v + uv + 4tu + 2Au + 1},
$$
$$
2A tu^2v + A^2 u^2v + Auv + 4A tu + 2A^2 u + A = 2tuv + Auv + v + 2t + A.
$$
Substituting $A = \sqrt{D} - t$ gives:
$$
(\sqrt{D} - t)^2 (u^2 v + 2u) + (\sqrt{D} - t)(2t u^2 v + 4 t u) = 2 t u v + v + 2t,
$$
$$
D = t^2 + \frac{2 t u v + v + 2 t}{u^2 v + 2 u}.
$$
Since $D, t, u, v$ are integers, we must have $u^2 v + 2 u \mid 2 t u v + v + 2 t$, and also $u \neq 0$ and $uv + 2 \neq 0$.

Therefore $u \mid 2 t u v + v + 2 t$, i.e., there exists $x \in \mathbb{Z}$ such that $t = \frac{ux - v}{2}$. Substituting into the original polynomial gives:
$$
D = \left( \frac{ux - v}{2} \right)^2 + \frac{u^2 v x - u v^2 + v + ux - v}{u(uv + 2)} = \left( \frac{ux - v}{2} \right)^2 + \frac{uvx - v^2 + x}{uv + 2}.
$$

Now, we must also have $uv + 2 \mid uvx + 2x - v^2 - x$, i.e., there exists $y \in \mathbb{Z}$ such that  $x = y(uv+2) - v^2$. Substituting this gives:
$$
D = \left( \frac{u(y(uv+2) - v^2) - v}{2} \right)^2 + y(uv+2) - v^2 - y. \qedhere
$$
\end{proof}

Using the expression that we just established, we can show our conclusion that all residues can be attained for a period of $4$.

\begin{theorem}\label{zbytky pro4}
  Let $n \in \N$ and $m \in \Z$. Then there exists a positive integer $D$ such that the continued fraction of $\sqrt{D}$ has a period length of $4$ and $D \equiv m \pmod{n}$.
\end{theorem}
\begin{proof}
    Using Proposition \ref{period 4}, we have:
    $$D=\left(\frac{u\left(y\left(uv+2\right)-v^2\right)-v}{2}\right)^2+y\left(uv+2\right)-v^2-y$$
    For $u\equiv v \equiv 0 \pmod {2n}$, we get:
    $$D \equiv 2y  - y = y  \pmod n.$$

    And since we have the freedom to choose $y$, for any choice of $u$ and $v$, we can select $y$ such that  $y  \equiv b\pmod  n$ for all $b\in\Z$. Hence, this polynomial attains all residues modulo $n$.
\end{proof}

\section{Continued fractions with general period length}\label{3}

Now we are going to work with general period lengths. Recall that 
when $\sqrt {D} = [a_0,\overline {a_1, \ldots, a_{k-1}, a_k}]$, then we define the sequences $p$ and $q$ recursively as: $$
    p_{-1} = 1, \quad 
    p_0 = a_0, \quad 
    p_n = a_np_{n-1} + p_{n-2}$$

    $$q_{-1} = 0, \quad 
    q_0 = 1, \quad 
    q_n = a_nq_{n-1} + q_{n-2}.$$

Let us also summarize Friesen's results, which we will use to express $D$ in terms of the coefficients.

\begin{theorem}[Friesen {\cite{friesen1988}}] \label{friesen}
The equation  
$$\sqrt{D} = [\lfloor\sqrt{D}\rfloor;\overline{ a_1, a_2, \ldots, a_{k-2} = a_2,\; a_{k-1} = a_1,\; a_k = 2\lfloor\sqrt{D}\rfloor}]$$  
has, for any symmetric sequence of positive integers $(a_1, \ldots, a_{k-1})$, infinitely many squarefree solutions $D$, whenever $q_{k-2}$ or $\tfrac{q_{k-2}^2 - (-1)^k}{q_{k-1}}$ is even. 

If both quantities are odd, then there are no solutions $D$ even if the squarefree condition is dropped.  
And $D = \alpha b^2 + \beta b + \gamma$ for any $b\in \Z $ sufficiently large so that $D>0$, where the coefficients $\alpha, \beta, \gamma$ are given below:

\begin{enumerate}
    \item  If
     $q_{k-1} \equiv 1 \pmod 2$  then: $$\alpha = q_{k-1}^2,$$
    $$\beta = 2q_{k-2}-\left(-1\right)^kq_{k-2}\left(q_{k-2}^2-\left(-1\right)^k\right),$$
    $$\gamma = \frac{\left(\frac{q_{k-2}^2}{4}-\left(-1\right)^k \right) \left(q_{k-2}^2-\left(-1\right)^k \right )^2}{q_{k-1}^2}.$$
    \item If $q_{k-1} \equiv 0 \pmod 2$  then: $$\alpha = \frac{q_{k-1}^2}{4},$$
    $$\beta = \frac{2q_{k-2}-\left(-1\right)^kq_{k-2}\left(q_{k-2}^2-\left(-1\right)^k\right)}{2},$$
    $$\gamma = \frac{\left(\frac{q_{k-2}^2}{4}-\left(-1\right)^k \right) \left(q_{k-2}^2-\left(-1\right)^k \right)^2}{q_{k-1}^2}.$$ 
\end{enumerate}

    The discriminant of the polynomial $D = \alpha b^2 + \beta b + \gamma$ can attain only the values $1, -1, 4$ and $-4$. Specifically, for odd $q_{k-1}$, the discriminant equals $4 \cdot (-1)^k$, and for even $q_{k-1}$, the discriminant equals $ (-1)^k$.
\end{theorem}

We now use these results to prove the statements about the congruence conditions on $D$ for continued fractions of $\sqrt{D}$ with a given period length. We consider two cases: even and odd period lengths. We begin with the case of an odd period length.

\begin{theorem}\label{4.5}
    Let $\sqrt{D}=[a_0, \overline{a_1,\ldots,a_k}]$, where $k\equiv1\pmod{2}$. Then $D\not\equiv0\pmod{p}$ for a prime $p\equiv 3\pmod{4}$.
\end{theorem}
\begin{proof}
    For the sake of contradiction, assume that there exists such $D\equiv 0 \pmod p$. Then we see that we can express $D$ as $D = \alpha b^2 + \beta b + \gamma\equiv 0\pmod{p}$. According to Theorem \ref{friesen}, the discriminant of this polynomial is equal to $-4$ or $-1$. Then, for the equation to have solutions, it follows that its discriminant has to be congruent to a square modulo $p$. Thus:
    $-4 \equiv c^2 \pmod p$ or $-1 \equiv c^2 \pmod p$.
    Since 4 and 1 are quadratic residues for every $p\equiv3\pmod{4}$, then in both cases $-1$ has to be a quadratic residue as well, which it is not. And we have the desired contradiction.
\end{proof}

Note that by this proof we have generalized theorems in the previous section, where we established the same property for $k=1$ (and $k=3$). 

We now continue with the case of an even period length. First, let us prove several useful lemmas that will be needed in the main proof. In the statements, $\alpha$ and $\beta$ are as in Theorem \ref{friesen}.

\begin{lemma}\label{alfa 0}
    Let $\sqrt{D}=[a_0, \overline{a_1,\ldots,a_k}]$, $k$ be an even number, $p$ an odd prime, and $a\in \N$. Let $l = \frac{k}{2}-1$. Then:
    \begin{enumerate}[label=\alph*)]
        \item if $l$ is odd, by choosing $a_l \equiv a_{l-2} \equiv a_{l-4} \dots \equiv a_1 \equiv 0 \pmod{p^a}$
        \item if $l$ is even, by choosing  $a_l \equiv a_{l-2} \equiv a_{l-4} \dots \equiv a_4 \equiv 0 \pmod{p^a}$, $a_1 \equiv x \not \equiv 0 \pmod{p^a}$, $a_2 \equiv -\frac{1}{x} \pmod{p^a}$
    \end{enumerate}
    we obtain $\alpha \equiv q_{k-1} \equiv 0 \pmod {p^a}$.
\end{lemma}

 \begin{proof}

We start with the well-known property:
 $$q_{k-1} = q_l q_{l+1} + q_{l-1} q_l = q_l\left(q_{l+1} + q_{l-1}\right).$$
 
\begin{enumerate}[label=\alph*)]
    \item By choosing $a_l \equiv a_{l-2} \equiv a_{l-4} \equiv \ldots \equiv a_1 \equiv 0 \pmod {p^a}$ for odd $l$, so {$q_1 = a_1 \equiv 0 \pmod {p^a}$} and by induction we obtain that if $q_{b-2} \equiv 0 \pmod p$, that also $q_{b} = a_{b} q_{b-1} + q_{b-2}  \equiv 0 \pmod p$ for odd $b \leq l$. 
    \item For even $l$ by choosing $a_l \equiv a_{l-2} \equiv a_{l-4} \equiv \ldots \equiv a_4 \equiv 0 \pmod {p^a}$, $a_1 \equiv x \not \equiv 0 \pmod {p^a}$, $a_2 \equiv -\frac{1}{x} \pmod {p^a}$, we obtain  $q_2 = a_2q_1 + q_0 = a_1a_2 + 1 \equiv 0 \pmod {p^a}$ and by induction we again obtain that if $q_{b-2} \equiv 0 \pmod p$, then also $q_{b} = a_{b} q_{b-1} + q_{b-2}  \equiv 0 \pmod p$ for even $b \leq l$.
\end{enumerate}

At the same time, from Theorem \ref{friesen}, any set $a_1, a_2, \ldots a_l, a_{l+1}$ has a solution if at least one of $q_{k-2}$ and $\frac{q_{k-2}^2-\left(-1\right)^k}{q_{k-1}}$ is even. We only got conditions modulo $p^a$ on $a_1, a_2, \ldots a_l, a_{l+1}$. It suffices to choose any suitable solution modulo 2. From the Chinese Remainder Theorem, we obtain conditions on $a_1, a_2, \ldots a_l, a_{l+1}$ modulo $2p^a$, and from Theorem \ref{friesen}, the given set $a_1, a_2, \ldots a_l, a_{l+1}$ has a solution.
\end{proof}

\begin{lemma} \label{beta}
    Let $\sqrt{D}=[a_0, \overline{a_1,\ldots,a_k}]$, where $k$ is an even number. Let $p$ be an odd prime, $a\in \N$, and $l = \frac{k}{2}-1$. Then:
    \begin{enumerate}[label=\alph*)]
        \item if $l$ is odd, by choosing $a_l \equiv a_{l-2} \equiv a_{l-4} \ldots \equiv a_1 \equiv 0 \pmod {p^a}$
        \item if $l$ is even, by choosing  $a_l \equiv a_{l-2} \equiv a_{l-4} \ldots \equiv a_4 \equiv 0 \pmod {p^a}$, $a_1 \equiv x \not \equiv 0 \pmod {p^a}$, $a_2 \equiv -\frac{1}{x} \pmod {p^a}$
    \end{enumerate}
    we obtain that $\beta$ and $p^a$ are coprime.
\end{lemma}
\begin{proof}
We know that: $$\beta = 2q_{k-2}-q_{k-2}\left(q_{k-2}^2 -1\right) = q_{k-2}\left(3 - q_{k-2}^2\right),$$
or
$$\beta = \frac{2q_{k-2}-q_{k-2}\left(q_{k-2}^2 -1\right)}{2} = \frac{q_{k-2}\left(3 - q_{k-2}^2\right)}{2}.$$
Simultaneously $q_{k-1} \equiv 0 \pmod {p^a}$. 
As $q_{k-2}$ and $q_{k-1}$ are coprime, $q_{k-2}$ and $p^a$ are also coprime. It suffices to prove that also $3 - q_{k-2}^2$ and $p^a$ are coprime, so that $p \not \mid 3 - q_{k-2}^2$.
\begin{enumerate}[label=\alph*)]
    \item  For odd $l$ we chose $a_{k-1} \equiv \ldots a_l \equiv a_{l-2} \equiv a_{l-4} \ldots \equiv a_1 \equiv 0 \pmod {p}$. So for any $b$ it holds that $q_b = a_b q_{b-1} + q_{b-2}\equiv q_{b-2} \pmod p$. By induction then $q_{k-2} \equiv q_2 = a_2a_1 + 1 \equiv 1 \pmod p$. 
\item For even $l$ we chose $a_{k-4}, \ldots , a_l \equiv a_{l-2} \equiv a_{l-4} \ldots \equiv a_4 \equiv 0 \pmod {p}$ a $a_{k-2} = a_2 \equiv -\frac{1}{a_1} \pmod {p}$. Then for any $b< k-2$ it holds that $q_b = a_bq_{b-1} + q_{b-2}\equiv q_{b-2} \pmod p$. By induction then $q_b \equiv q_1 = a_1 \pmod p$. Simultaneously for even $2<b<k-2$ it holds that $q_b = a_bq_{b-1} + q_{b-2}\equiv q_{b-2} \pmod p$. By induction then $q_b \equiv q_2 = 0 \pmod p$. Further,  $q_{k-2} = a_2q_{k-3} + q_{k-4} \equiv -\frac{1}{a_1}a_1 + q_{k-4} \equiv -1 \pmod p $. 
\end{enumerate}
In both cases then $q_{k-2}^2 \equiv 1 \pmod p$. For $p \not = 2$ that means $q_{k-2}^2 \not \equiv 3$ and $p \nmid 3 - q_{k-2}^2$. And $p^a$ and $\beta$ are then coprime. 
\end{proof}

We are now ready to present the proof for the case of an even period length.
\begin{theorem} \label{modulo n}
   Let $m,n,k$ be positive integers, where $n$ is odd and $k$ is even. Then there exists a positive integer $D$ such that the continued fraction $\sqrt{D}$ is of the form $\sqrt{D}=[a_0, \overline{a_1,\ldots,a_k}]$ and $D\equiv m\pmod n$.
    
\end{theorem}

\begin{proof}
    We will use the same choice of $a_l, a_{l-2}, a_{l-4} \ldots $ as in Lemmas \ref{alfa 0} and \ref{beta}. From Lemma \ref{alfa 0}, we can obtain that $\alpha \equiv q_{k-1} \equiv 0 \pmod {p^a}$ for any odd prime $p$ and $a\in\Z^+$ by choosing $a_l, a_{l-2}, a_{l-4} \ldots $ modulo $p^a$, where $l = \frac{k}{2}-1$. For $n = r_1^{a_1}r_2^{a_2}\cdots r_m^{a_m}$, where $r_1,...,r_m$ are distinct primes, we can choose $a_l, a_{l-2}, a_{l-4} \ldots $ modulo $r_1^{a_1}, r_2^{a_2}, \ldots , r_m^{a_m}$, and using the Chinese Remainder Theorem, we then obtain $\alpha \equiv q_{k-1} \equiv 0 \pmod {n}$. At the same time, from Lemma \ref{beta}, we know that for every $p^a\mid n$, it holds that $p \nmid \beta$. Thus, $n$ and $\beta$ are coprime. Then,
    $D = \alpha b^2 + \beta b + \gamma \equiv \beta b + \gamma \pmod {n}$, and due to $n$ and $\beta$ being coprime, $D$ with a free variable $b$ attains every value modulo $n$.
\end{proof}
For odd $k$, we proved that $D$ cannot be divisible by any prime $p\equiv 3 \pmod 4$, so the assumption of an even $k$ is necessary.  

For even $n$, using the choice from Lemma \ref{alfa 0}, we cannot ensure that in the case a), the parity conditions from Friesen  \ref{friesen} will be satisfied and there might not exist any solution for $D$. There would be a problem with even $n$ in  Lemma \ref{beta} as well because $3-q_{k-2}^2\equiv2\equiv0\pmod2$. It does not mean that the theorem does not hold for even $n$ because there could exist a different choice of $a_1,...,a_k$ that satisfies all conditions.

\section{Universal quadratic forms}\label{4}
Let us recall the definition of a quadratic form and a universal quadratic form. Throughout this section, let $D \in \Z_{>1}$, which is not a perfect square. As a slight simplification, we will work with the order $\Z[\sqrt{D}]$. Of course, this is the full ring of integers when $D\equiv 2,3\pmod 4$.

\begin{definition}
    Let $\alpha = a + b\sqrt{D} \in \mathbb{Z}[\sqrt{D}]$. Then $\alpha' = a - b\sqrt{D}$ is called its conjugate. We say $\alpha$ is totally positive, denoted $\alpha \succ 0$, if $\alpha > 0$ and $\alpha' > 0$.
\end{definition}

\begin{definition}[Quadratic form]
    By a quadratic form $Q(x_1, x_2, \ldots, x_r)$ over $\mathbb{Z}[\sqrt{D}]$, we mean
        $$\sum_{1 \leq i \leq j \leq r} \alpha_{ij} x_i x_j,$$
    where $\alpha_{ij} \in \mathbb{Z}[\sqrt{D}]$ and $x_1, x_2, \ldots, x_r \in \mathbb{Z}[\sqrt{D}]$ are variables.
\end{definition}
\begin{definition}
    A totally positive quadratic form is a quadratic form 
    $Q(x_1, x_2, \ldots, x_r)$ over $\mathbb{Z}[\sqrt{D}]$ that satisfies 
    $Q(\gamma_1, \gamma_2, \ldots, \gamma_r) \succ 0$ for all 
    $\gamma_1, \gamma_2, \ldots, \gamma_r \in \mathbb{Z}[\sqrt{D}]$, except in the case 
    $Q(0,0,\ldots,0) = 0$.
\end{definition}

\begin{definition}[Universal quadratic form]
    A universal quadratic form over $\mathbb{Z}[\sqrt{D}]$ is a totally positive quadratic form 
    $Q(x_1, x_2, \ldots, x_r)$ over $\mathbb{Z}[\sqrt{D}]$ such that for every 
    $\alpha \succ 0$ there exist elements 
    $\gamma_1, \gamma_2, \ldots, \gamma_r \in \mathbb{Z}[\sqrt{D}]$ satisfying 
       $$ Q(\gamma_1, \gamma_2, \ldots, \gamma_r) = \alpha.$$
\end{definition}

We also state the result of Kala and Tinková \cite{kalatinkova2023} concerning the number of variables required for a universal quadratic form. Note that although \cite{kalatinkova2023, Kala_2023} states this result for the full ring of integers, it holds with exactly the same proof for the order $\Z[\sqrt{D}]$, even when it is not maximal.

\begin{theorem}[{Kala--Tinková, see \cite[Theorem 5.5]{Kala_2023}}]\label{promenny}
    Let $\sqrt{D} = [a_0, \overline{a_1, \ldots a_k}]$ and
    \[U = \begin{cases}
        \max (a_1, a_3, \ldots, a_{k-1}), & \text{if $k$ is even}\\
        \sqrt{D}, & \text{if $k$ is odd}
    \end{cases}\]
    Let $Q$ be universal quadratic form over  $\Z[\sqrt{D}]$ of rank $r$.
    \begin{enumerate}[label=\alph*)]
        \item If $Q$ is classical, then $r \geq \frac{U}{2}$.
        \item In general $r \geq \frac{\sqrt{U}}{2}$ (assuming $U \geq 240$). 
    \end{enumerate}
\end{theorem}

We can now use our results from the previous sections, together with this theorem, to prove that there exists $D$ in a given congruence class, such that universal quadratic forms over $\Z[\sqrt{D}]$ must have a lot of variables.
\begin{theorem} \label{nej1}
    For every $m, n, s\in \Z$ there exists a non-square $D\geq2$ such that $D \equiv m \pmod n$, and every universal quadratic form over $\Z[\sqrt{D}]$ has at least $s$ variables. 
\end{theorem}

\begin{proof}
    According to Theorem \ref{zbytky pro4}, $D$ can, for $\sqrt{D} = [t, \overline{u,v,u, 2t}]$, attain every value modulo any $n \in \N$, and this imposes only congruence conditions on $t, u$ and $v$ modulo $2n$ (or $n$), so the coefficients $a_1=a_3=u$ can be chosen arbitrarily large. Then it suffices to choose $u$ such that $\frac{\sqrt {u}}{2}\geq s$ and simultaneously $\max(a_1=u, a_3=u) =u\geq 240$. From Theorem \ref{promenny}, it then follows that the number of variables of any universal quadratic form over $Z[\sqrt{D}]$ is at least: \begin{align*}r \geq \frac{\sqrt {U}}{2} = \frac{\sqrt {\max(a_1, a_3, \ldots , a_{k-1})}}{2}&=\frac{\sqrt{u}}{2}\geq s.\qedhere
\end{align*}
\end{proof}
The following theorem additionally allows us to prescribe an arbitrary even period length.

\begin{theorem}\label{nej2}
    For every $m,s \in \Z$, odd $n$, and even $k>0$, there exists a non-square $D\geq2$ such that $D \equiv m \pmod n$, continued fraction $\sqrt{D}$ has a period length of $k$, 
    and every universal quadratic form over $\Z[\sqrt{D}]$ has at least $s$ variables. 
\end{theorem}
\begin{proof}
  From Theorem \ref{modulo n}, we know that we can choose $a_0, \ldots, a_k $ modulo $n$ such that $D \equiv m \pmod n$ for any $n$. Because these are just modular conditions, we can choose  $a_0, \ldots , a_k$ arbitrarily large so that $\frac{\sqrt {\max(a_1, a_3, \ldots , a_{k-1})}}{2}\geq s$ and at the same time \mbox{$\max(a_0, \ldots , a_k) \geq 240.$}. Then by Theorem \ref{promenny}, the number of variables of any universal quadratic form over $\Z[\sqrt{D}]$ is at least: \begin{align*}
r \geq \frac{\sqrt {U}}{2} &= \frac{\sqrt {\max(a_1, a_3, \ldots , a_{k-1})}}{2}\geq s. \qedhere
\end{align*}
\end{proof}

\bibliographystyle{plain}
\bibliography{literatura}

\end{document}